\documentclass[10pt]{amsart}
\usepackage[]{fontenc}
\usepackage[latin9]{inputenc}
\usepackage[letterpaper]{geometry}
\usepackage{amsthm}
\usepackage{setspace}
\usepackage{amssymb}
\usepackage{esint}
\makeatletter
\theoremstyle{plain}
\newtheorem{thm}{Theorem}
  \theoremstyle{definition}
  \newtheorem{defn}[thm]{Definition}
  \theoremstyle{plain}
  \newtheorem{prop}[thm]{Proposition}
  \theoremstyle{plain}
  \newtheorem{lem}[thm]{Lemma}
  \theoremstyle{plain}
  \newtheorem{cor}[thm]{Corollary}
  \theoremstyle{remark}
  \newtheorem{claim}[thm]{Claim}

\DeclareMathOperator{\lip}{Lip}

\DeclareMathOperator{\bdim}{bdim}
\DeclareMathOperator{\diam}{diam}
\DeclareMathOperator{\sep}{sep}

\numberwithin{thm}{section}
\numberwithin{equation}{section}

\makeatother

\begin{document}

\title{Slow entropy and differentiable models for infinite-measure preserving
$\mathbb{Z}^{k}$ actions }

\author{Michael Hochman}

\thanks{Supported by NSF grant 0901534. }

\address{Fine Hall, Washington Rd, Princeton NJ 08544}

\email{hochman@math.princeton.edu}

\dedicatory{In memory of Dan Rudolph}


\begin{abstract}
We define {}``slow'' entropy invariants for $\mathbb{Z}^{d}$ actions
on infinite measure spaces, which measures growth of itineraries at
subexponential scales. We use this to construct infinite-measure preserving
$\mathbb{Z}^{2}$ actions which cannot be realized as a group of diffeomorphisms
of a compact manifold preserving a Borel measure, contrary to the
situation for $\mathbb{Z}$-actions, where every infinite-measure
preserving action can be realized in this way. 
\end{abstract}
\maketitle
\markboth{Michael Hochman}{Slow entropy for infinite measure preserving actions}

\section{\label{sec:Introduction}Introduction}

Let $T=(T^{u})_{u\in\mathbb{Z}^{k}}$ be a finite-measure preserving
(f.m.p.) $\mathbb{Z}^{k}$-action on a Lebesgue space $(\Omega,\mathcal{B},\mu)$.
We always assume the action is ergodic and free, and for simplicity
assume that the total mass is $\mu(\Omega)=1$. It is a classical
problem to determine when such an action has a differentiable model,
i.e. when it is isomorphic to the action of a group of diffeomorphisms
on a compact manifold preserving a Borel measure (and, more specifically,
when it has a smooth model, i.e. a differentiable model in which the
measure is absolutely continuous with respect to the volume). It is
well known that entropy presents various obstructions: a f.m.p. $\mathbb{Z}$-action
has a differentiable model if and only if the entropy is finite (whether
this suffices for a smooth model is a longstanding open question).
Sufficient conditions for a f.m.p. $\mathbb{Z}^{k}$ action to have
a differentiable model are not known when $k\geq2$, but a necessary
condition is that the entropy must be $0$, and other obstructions
of entropy type have also been identified, which we shall discuss
further below.

In this paper we investigate the existence of differentiable models
for infinite-measure preserving (i.m.p.) actions, that is, actions
such as above but with $\mu$ an infinite $\sigma$-finite measure.
For $\mathbb{Z}$-actions this question is trivial: \emph{every }ergodic
i.m.p. $\mathbb{Z}$-action has a differentiable model. Indeed, by
a theorem of Krengel, such actions have a two set generator \cite{Krengel970},
and hence one can transfer the measure to a horseshoe, giving a differentiable
version of the action (the existence of smooth models is again open,
for a discussion of the non-singular case, see \cite[Section 7]{DanilenkoSilva2008}).
The main result of this paper is that, for higher rank i.m.p. actions,
existence of differentiable models is not automatic: 
\begin{thm}
\label{thm:non-realizable actions-exist}There exist ergodic i.m.p.
$\mathbb{Z}^{2}$-actions without a differentiable model.
\end{thm}
The mechanism which underlies Theorem \ref{thm:non-realizable actions-exist},
as well as the classical results for f.m.p. actions mentioned above,
is, briefly, the following (see Section \ref{sec:Lipschitz-actions}
for more details). For a compact metric $d(\cdot,\cdot)$ on $\Omega$,
write \[
\sep(\Omega,d,\varepsilon)=\max\left\{ N\,\left|\begin{array}{c}
\exists\, x_{1},\ldots,x_{N}\in\Omega\mbox{ such that}\\
d(x_{i},x_{j})\geq\varepsilon\mbox{ for }1\leq i<j\leq N\end{array}\right.\right\} \]
For a $\mathbb{Z}^{k}$-action $T$ on $\Omega$ and $n\in\mathbb{N}$,
denote the Bowen metric by \[
d_{n}^{\infty}(x,y)=\sup_{\left\Vert u\right\Vert \leq n}d(T^{u}x,T^{u}y)\]
If the action is by Lipschitz maps and $(\Omega,d)$ has finite Minkowski
(box) dimension, as is the case for actions by diffeomorphisms on
compact manifolds, then it is simple to show that $\sep(\Omega,d_{n}^{\infty},\varepsilon)$
grows at most exponentially in $n$ for every $\varepsilon>0$. 

In contrast, dynamical invariants of entropy type give, for small
$\varepsilon>0$, lower bounds for the growth of $\sep(\Omega,d_{n}^{\infty},\varepsilon)$
as $n\rightarrow\infty$. For example, if a f.m.p. $\mathbb{Z}^{k}$
action has positive entropy then, when $\varepsilon$ is sufficiently
small, $\sep(\Omega,d_{n}^{\infty},\varepsilon)\geq c\exp(cn^{k})$
for some $c>0$, and hence, when $k>1$, this is an obstruction to
differentiable realization. A finer invariant, introduced by Katok
and Thouvenot under the name {}``slow entropy'' \cite{KatokThouvenot1997},
has the property that when its value for a f.m.p. action is $>\alpha$,
one is guaranteed for small $\varepsilon$ that $\sep(\Omega,d_{n}^{\infty},\varepsilon)\geq c\exp(n^{\alpha})$
along some subsequence as $n\rightarrow\infty$. 

The proof of Theorem \ref{thm:non-realizable actions-exist} follows
a similar strategy. The first step is to define a {}``slow entropy''
invariant $\rho_{slow}(T,\mu)$ of i.m.p. $\mathbb{Z}^{k}$-actions.
The definition is a variation on the standard name-counting definition
of entropy, using a stretched-exponential scale. Since the details
are important for the discussion that follows, we give the complete
definition here. 

Let $\mathcal{P}$ be a finite, measurable partition of $\Omega$.
We restrict the discussion to \emph{co-finite }partitions, i.e. we
assume that all atoms but one have finite measure. The union of the
finite-measure atoms is called the \emph{core}.%
\footnote{In order to apply what follows to f.m.p. actions, assume instead that
all atoms of $\mathcal{P}$ are finite, and then the core is the entire
space. %
} Given $n\in\mathbb{N}$ let \[
Q_{n}=\{u\in\mathbb{Z}^{k}\,:\,\left\Vert u\right\Vert _{\infty}\leq n\}\]
For $x,y\in\Omega$, the usual way to compare the partial orbits of
$x,y$ on $Q_{n}$ is by the Hamming distance between their $\mathcal{P}$-itineraries,
i.e., the proportion of $u\in Q_{n}$ such that $T^{u}x,T^{u}y$ belong
to different atoms of $\mathcal{P}$. However, since the core of $\mathcal{P}$
has finite measure and $\mu(\Omega)=\infty$, the ergodic theorem
tells us that the orbits of $x,y$ typically spend a 0-fraction of
their time in the core, and hence almost all their time in the common
atom of infinite measure. Consequently the Hamming distance tends
to $0$ as $n\rightarrow\infty$, and $\mu$-a.a. pair $x,y$ is Hamming-asymptotic.
Instead, we introduce a Hamming-like distance {}``relative'' to
the visits to the core: \begin{equation}
d_{\mathcal{P},n}(x,y)=\frac{\#\{u\in Q_{n}\,:\, T^{u}x,T^{u}y\mbox{ are in different }\mathcal{P}\mbox{-atoms}\}}{\#\{u\in Q_{n}\,:\, T^{u}x\mbox{ or }T^{u}y\mbox{ are in the core of }\mathcal{P}\}}\label{eq:Pn-distance}\end{equation}
With the convention $0/0=0$ this is a pseudo-metric on $\Omega$
(Lemma \ref{lem:pseudo-metric-verification}). We denote the $d_{\mathcal{P},n}$-diameter
of $E\subseteq A$ by \[
\diam(E,d_{\mathcal{P},n})=\sup_{x,y\in E}d_{\mathcal{P},n}(x,y)\]

Next, given $\varepsilon>0$, the usual definition of entropy counts
the number of sets of $d_{\mathcal{P},n}$-diameter $\varepsilon$
that are needed to cover all but an $\varepsilon$-fraction of the
space. Since in our setting $\mu(\Omega)=\infty$, this does not makes
sense. Instead we fix a reference set $A\subseteq\Omega$ of positive
and finite measure, and define the $\varepsilon$-covering number
by\begin{equation}
N(A,d_{\mathcal{P},n},\varepsilon)=\min\left\{ N\,\left|\begin{array}{c}
\exists\; E_{1},\ldots,E_{N}\subseteq\Omega\mbox{ with }\diam(E_{i},d_{\mathcal{P},n})\leq\varepsilon\\
\mbox{ and }\mu(A\cap\bigcup_{i=1}^{N}E_{i})\geq(1-\varepsilon)\mu(A)\end{array}\right.\right\} \label{eq:epsilon-covering-numbers}\end{equation}
Note that $N(A,d_{\mathcal{P},n},\varepsilon)$ is monotone in $\varepsilon$,
and we may define \begin{equation}
\rho_{slow}(T,\mu,A,\mathcal{P})=\lim_{\varepsilon\searrow0}\left(\limsup_{n\rightarrow\infty}\frac{\log(\log N(A,d_{\mathcal{P},n},\varepsilon))}{\log n}\right)\label{eq:slow-scale}\end{equation}
Note that by dividing by $\log n$ rather than $\log|Q_{n}|=k\log(2n+1)$
we have normalized $\rho_{slow}$ so that $0\leq\rho_{slow}(T,\mu)\leq k$.
We shall see later that $\rho_{slow}(T,\mu,A,\mathcal{P})$ does not
depend on the choice of $A$ (Corollary \ref{cor:covering-rho-entropy-on-diffferent-sets}),
and we denote it $\rho_{slow}(T,\mu,\mathcal{P})$. Thus, roughly
speaking, $\rho_{slow}(T,\mu,\mathcal{P})=\alpha$ means that, when
$\varepsilon$ is small, $N(A,d_{\mathcal{P},n},\varepsilon)$ grows
like $2^{n^{\alpha}}$ along some subsequence. Finally, the \emph{slow
entropy }of $(\Omega,\mathcal{B},\mu,T)$ is \[
\rho_{slow}(T,\mu)=\sup\rho_{slow}(T,\mu,\mathcal{P})\]
where the supremum is over co-finite partitions $\mathcal{P}$. 

The quantity $\rho_{slow}(T,\mu)$ is clearly an isomorphism invariant,
and it reduces to Katok-Thouvenot slow entropy when applied to f.m.p.
$\mathbb{Z}^{k}$-actions (to do so we drop the co-finiteness condition
on partitions). It also shares several basic features with entropy.
In particular, we shall see that it can be computed from a generating
co-finite partition.

Returning to our original problem, the second step in the proof of
Theorem \ref{thm:non-realizable actions-exist} is to relate $\rho_{slow}(T,\mu)$
to the growth of $\sep(\Omega,d_{n}^{\infty},\varepsilon)$ when $\Omega$
is endowed with a metric. One might expect, as in the f.m.p. case,
that $\rho_{slow}(T,\mu)>1$ would imply super-exponential growth,
and hence be an obstruction to differentiable realization. While it
is possible that this is true, we have not been able to prove it.
Instead, we introduce the following notion:
\begin{defn}
\label{def:pure-actions}An action $(\Omega,\mathcal{B},\mu,T)$ has
uniform slow entropy $\alpha$ if $\rho_{slow}(T,\mu,\mathcal{P})=\alpha$
for every non-trivial co-finite partition $\mathcal{P}$. 
\end{defn}
Equivalently, there is no factor with lower slow entropy. The condition
of uniform slow entropy is similar to uniform entropy dimension as
defined by Dou, Huang and Park \cite{DouHuangPark2011} (entropy dimension
is an invariant of f.m.p. $\mathbb{Z}$-actions which measures subexponential
growth of the entropy of partitions refined along subsequences, but
appears not to be equivalent to slow entropy). uniform $\alpha$-slow
entropy may be seen as generalizing completely positive entropy, since
f.m.p. $\mathbb{Z}^{k}$-actions of completely positive entropy have
uniform slow entropy $k$.
\begin{thm}
\label{thm:large-s-implies-no-diff-model} Suppose that $(\Omega,\mathcal{B},\mu,T)$
is an ergodic $\mathbb{Z}^{k}$-action by Lipschitz maps on a complete
metric space of finite box dimension, preserving a Borel measure.
If the action does not have uniform slow entropy then $\rho_{slow}(T,\mu)\leq1$.
\end{thm}
This is proved in Section \ref{sub:Lipshifts-actions-dimension-and-entropy}.
Following this we derive Theorem \ref{thm:non-realizable actions-exist}
by constructing, via cutting and stacking, a $\mathbb{Z}^{2}$-action
which does not have uniform slow entropy and with $\rho_{slow}(T,\mu)>1$.

Having defined slow entropy for i.m.p. actions, let us now say a few
words and make some speculation about its relation to entropy theory.
In fact, the literature already contains at least four definitions
of entropy for i.m.p. actions: Krengel entropy \cite{Krengel1967},
Parry entropy \cite{Parry1969}, the entropy of Danilenko and Rudolph
\cite{DanilenkoRudolph2009}, and Poisson entropy \cite{JanvresseMeyerovitchRoy2010}.
These are not believed to be equivalent, although inequivalence has
only been established between Krengel and Poisson entropies \cite{JanvresseDeLaRue2010}. 

Since $\rho_{slow}$ does not aim to capture exponential growth of
orbits, it should not be compared to invariants which do. Instead
one may ask whether it is related to one of the existing entropy theories
in the same way that Katok-Thouvenot slow entropy is related to Kolmogorov-Sinai
entropy. That is, if the definition of $\rho_{slow}$ is modified
to measure growth of orbits at the exponential scale, does the resulting
invariant coincide with one of the existing entropy theories for i.m.p.
actions? 

The obvious modification to make is to replace the ratio in the limit
\eqref{eq:slow-scale} with the quantity $|Q_{n}|^{-1}\cdot\log(N(A,\mathcal{P},n,\varepsilon))$.
If we do so, however, the resulting invariant $\rho_{exp}$ is trivial,
assigning the value is $0$ to every i.m.p. action. One way to understand
this fact is simply that i.m.p. actions are best understood as analogues
of zero entropy systems, and that the subexponential scale is the
appropriate one with which to study orbit growth for such actions.

Instead, there is a more interesting modification which involves a
re-scaling of time along orbits. The idea presented next is very close
to the entropy theory for cross sections developed recently by N.
Avni \cite{Avni2010}, which deals with countable probability-preserving
equivalence relations endowed with a cocycle into an amenable group.
In our setting such a relation arises by restricting the orbit relation
and orbit cocycle to a set $A$ of finite measure, and the mean ergodic
theorem for cross sections in the sense of \cite[Theorem 2.6]{Avni2010}
follows from the ratio ergodic theorem \cite{Hochman2009e}. However,
we shall present a more concrete version of the idea adapted to our
notation.

As we have already observed, when one considers the partial orbit
$(T^{u}x)_{u\in Q_{n}}$ for a large $n$, the frequency of visits
to the core is asymptotically negligible, and furthermore this frequency
depends on $x$. Instead, one can choose $Q_{n}$ in a manner depending
on $x$ so that the number of visits is approximately constant. More
precisely, for $x\in\Omega$ and $n\in\mathbb{N}$ define $m=m(x,n)$
to be the smallest integer such that there are at least $|Q_{n}|$
elements $u\in Q_{m}$ for which $T^{u}x$ is in the core of $\mathcal{P}$.
Define a pseudo-metric $\widetilde{d}_{\mathcal{P},n}$ on $\Omega$
by\[
\widetilde{d}_{\mathcal{P},n}(x,y)=d_{\mathcal{P},m(x,n)}(x,y)+d_{\mathcal{P},m(y,n)}(x,y)\]
Thus, we are comparing portions of the itineraries of $x,y$ which
have a similar number $|Q_{n}|$ of visits to the core. 

Proceeding as before but using the metric $\widetilde{d}_{\mathcal{P},n}$,
we can now define invariants $\widetilde{\rho}_{exp}$ and $\widetilde{\rho}_{slow}$
using, respectively, exponential and stretched-exponential scales. 

It turns out that $\widetilde{\rho}_{exp}$ is, for $\mathbb{Z}$-actions,
nothing other than Krengel entropy. Let us recall the definition.
Assuming the action is generated by the transformation $T:\Omega\rightarrow\Omega$
and $A$ is a set of positive finite measure, let $\mu_{A}$ denote
the normalized restriction of $\mu$ to $A$ and let $T_{A}:A\rightarrow A$
be the first return map to $A$, i.e. $T_{A}x=T^{r(A,x)}x$ where
$r(A,x)=\min\{i>0\,:\, T^{i}x\in A\}$. Krengel entropy is then defined
by\[
h_{Kr}(T,\mu)=\mu(A)h(T_{A},\nu_{A})\]
(this is independent of $A$ by Abramov's formula). 

The equivalence follows from Avni's work, and we shall not prove it
here. We note however that $\widetilde{\rho}_{exp}$ can be taken
as a definition of Krengel entropy for $\mathbb{Z}^{k}$ actions,
$k>1$; a new definition is necessary because there is no good notion
of an induced map in the higher rank case. Avni's machinery works,
in fact, for general discrete amenable groups (and some non-discrete
ones) and it would give a definition of Krengel entropy for actions
of those groups if and when a ratio ergodic theorem becomes available
for them.

Thus, $\widetilde{\rho}_{slow}$ is to Krengel entropy as Katok-Thouvenot
slow entropy is to Kolmogorov-Sinai entropy. This raises the question,
what is the relation between $\rho_{slow}$ and $\widetilde{\rho}_{slow}$?
One might expect that the relation should in some way involve the
recurrence properties of the action, since the difference in their
definitions is a re-scaling of time along orbits which regularizes
the frequency of returns to the partition's core. 

One almost trivial relation between recurrence properties and slow
entropy is the following. The recurrence set of $x$ to $A$, up to
{}``time'' $n$, is \begin{equation}
R_{n}(A,x)=\{u\in Q_{n}\,:\, T^{u}x\in A\}\label{eq:recurrence-set}\end{equation}
Define\begin{equation}
\alpha(A,x)=\limsup_{n\rightarrow\infty}\frac{\log|R_{n}(A,x)|}{\log|Q_{n}|}\label{eq:alpha}\end{equation}
Hence $0\leq\alpha\leq1$, and $\alpha$ is the largest number such
that $|R_{n}(A,x)|\approx|Q_{n}|^{\alpha+o(1)}$ along some subsequence.
By the Chacon-Ornstein lemma for $\mathbb{Z}^{k}$ actions \cite{Hochman2009e},
$\alpha(A,\cdot)$ is invariant, i.e. $\alpha(A,x)=\alpha(A,T^{u}x)$
for all $u\in\mathbb{Z}^{k}$, so by ergodicity $\alpha(A,x)$ is
a.s. independent of $x$. Also, given another set $B$ of finite measure,
the ratio ergodic theorem implies that $|R_{n}(A,x)|/|R_{n}(B,x)|\rightarrow\mu(A)/\mu(B)$,
hence the a.s. value of $\alpha(A,x)$ is also independent of $A$.
This justifies defining the \emph{recurrence dimension }$\alpha(T,\mu)$
of the action to be this value. We then have the following combinatorial
bound:
\begin{prop}
\label{pro:recurrence-bound-for-rho-slow}$\rho_{slow}(T,\mu)\leq k\alpha(T,\mu)$.
\end{prop}
We shall see in Section \ref{sub:a-rank-1-action} that strict inequality
is possible, but in view of the discussion above, one might expect
a more precise relationship such as $\rho_{slow}(T,\mu)=\alpha(T,\mu)\cdot\widetilde{\rho}_{slow}(T,\mu)$,
or perhaps at least an inequality between these quantities. Alternatively,
one might expect at least that $0<\widetilde{\rho}_{exp}(T,\mu)<\infty$
implies $\rho_{slow}(T,\mu)=k\cdot\alpha(T,\mu)$. A related phenomenon
was observed in the work of Galatolo, Kim and Park \cite{GalatoloKimPark2006},
who studied the relation of recurrence rates and Krengel entropy in
i.m.p. $\mathbb{Z}$-actions. We leave this matter for future investigation,
but note that, in any event, $\rho_{slow}$ seems better suited to
our application than $\widetilde{\rho}_{slow}$.

\subsection*{\emph{Acknowledgment}}

I am grateful to Nir Avni for some interesting discussions. This work
was completed during a visit to the Theory Group at Microsoft Research
in Redmond, Wa., and I would like to thank my hosts for their hospitality
and support.

\section{\label{sec:rho-entropy}Slow entropy}

\subsection{\label{sub:ratio-ergodic-thm}The ratio ergodic theorem}

For basic background on ergodic theory and infinite ergodic theory,
see \cite{Walters82} and \cite{Aaronson97}, respectively. We recall
the following basic fact, which was proved in \cite{Hochman2009e}:
\begin{thm}
\label{thm:ratio-ergodic-thm} If $(\Omega,\mathcal{B},\mu,(T_{u})_{u\in\mathbb{Z}^{k}})$
is an ergodic i.m.p. action, then for every $f,g\in L^{1}(\mu)$ with
$\int g\ d\mu\neq0$,\[
\lim_{n\rightarrow\infty}\frac{\sum_{u\in Q_{n}}f(T^{u}x)}{\sum_{u\in Q_{n}}g(T^{u}x)}=\frac{\int f\, d\mu}{\int g\, d\mu}\qquad\mu\mbox{-a.e.}\]

\end{thm}
In particular if $A,B$ have positive finite measure then, taking
$f=1_{A}$ and $g=1_{B}$, we have\[
\lim_{n\rightarrow\infty}\frac{|\{u\in Q_{n}\,:\, T^{u}x\in A\}|}{|\{u\in Q_{n}\,:\, T^{u}x\in B\}|}=\frac{\mu(A)}{\mu(B)}\]
Thus, the relative frequency with which an orbit visits $A$ and $B$
is equal to their relative masses. Note that for finite $\mu$ this
result is an easy consequence of the pointwise ergodic theorem, which
states that the frequency of visits to \emph{each }of the sets is
asymptotically equal to the mass of the set. This reasoning is not
valid in the infinite measure case; indeed, the frequency of visits
to a finite measure set is asymptotically zero. To see this, let $A$
be fixed and choose a sequence of sets $B_{k}$ with $\mu(B_{k})\rightarrow\infty$.
Then for each $n$,\[
\limsup_{n\rightarrow\infty}\frac{|\{u\in Q_{n}\,:\, T^{u}x\in A\}|}{|Q_{n}|}\leq\lim_{n\rightarrow\infty}\frac{|\{u\in Q_{n}\,:\, T^{u}x\in A\}|}{|\{u\in Q_{n}\,:\, T^{u}x\in B_{k}\}|}=\frac{\mu(A)}{\mu(B_{k})}\]
Taking $k\rightarrow\infty$ we find that the frequency of visits
to $A$ is zero.

A related result, which implies the ratio ergodic theorem, is the
Chacon-Ornstein lemma \cite{ChaconOrnstein60} which was proved for
$\mathbb{Z}^{k}$-actions in \cite{Hochman2009e}:
\begin{thm}
\label{thm:Chacon-Ornstein} Let $(\Omega,\mathcal{B},\mu,T)$ be
an i.m.p. $\mathbb{Z}^{k}$-action. Then for any non-negative $0\neq f\in L^{1}(\mu)\cap L^{\infty}(\mu)$,\[
\lim_{n\rightarrow\infty}\frac{\sum_{u\in Q_{n+1}\setminus Q_{n}}f(T^{u}x)}{\sum_{u\in Q_{n}}f(T^{u}x)}=0\qquad\mu\mbox{-a.e.}\]

\end{thm}

\subsection{\label{sub:Definition-of-the-invariants}The metric}

Let $(\Omega,\mathcal{B},\mu,T)$ be an i.m.p. $\mathbb{Z}^{k}$-action.
For a co-finite partition $\mathcal{P}=\{P_{1},\ldots,P_{r}\}$ we
assume by convention that $P_{1}$ is the infinite atom, so $\mu(P_{i})<\infty$
for $i=2,\ldots,r$ and $\bigcup_{i=2}^{r}P_{i}$ is the core of $\mathcal{P}$. 

For $x\in\Omega$, write $\mathcal{P}(x)=i$ if $x\in P_{i}$, and
given $n\in\mathbb{N}$ denote by $x_{\mathcal{P},n}\in\{1,\ldots,r\}^{Q_{n}}$
the $(\mathcal{P},n)$-name $x$, i.e., the coloring $u\mapsto\mathcal{P}(T^{u}x)$
of $Q_{n}$. It is convenient to introduce a distance $d(a,b)$ for
$a,b\in\{1,\ldots,r\}^{Q_{n}}$ given by\begin{equation}
d(a,b)=\frac{|\{u\in Q_{n}\,:\, a_{u}\neq b_{u}\}|}{|\{u\in Q_{n}\,:\, a_{u}\neq1\mbox{ or }b_{u}\neq1\}|}\label{eq:metric-on-patterns}\end{equation}
using again the conventions that $0/0=0$. Then $d(a,b)\leq1$, and
the distance $d_{\mathcal{P},n}(\cdot,\cdot)$ from the introduction
is nothing other than $d_{\mathcal{P},n}(x,y)=d(x_{\mathcal{P},n},y_{\mathcal{P},n})$.
\begin{lem}
\label{lem:pseudo-metric-verification} $d$ is a metric and $d_{\mathcal{P},n}$
is a pseudo-metric. \end{lem}
\begin{proof}
The second statement follows from the first. Positivity and symmetry
of $d$ are immediate, so it remains to show, for $a,b,c\in\{1,\ldots,r\}^{Q_{n}}$,
that $d(a,c)\leq d(a,b)+d(b,c)$.

We can assume that $b\neq a$, since otherwise the inequality is trivial.
Let $A=\{u\in Q_{n}\,:\, a_{u}\neq1\}$ and define $B,C$ similarly
using $b,c$, respectively. Clearly\begin{eqnarray*}
d(a,b) & = & \frac{|A\setminus B|+|\{u\in B\,:\, a_{u}\neq b_{u}\}|}{|A\cup B|}\end{eqnarray*}
and similarly for $d(b,c)$. 

Suppose first that $B\subseteq A\cup C$. We then have Since $|A\cup C|\geq|A\cup B|$,
the inequality $d(a,c)\leq d(a,b)+d(b,c)$ will follow from the inequality
for the numerators of the corresponding expressions above for $d(a,b),d(b,c)$,
i.e. from\begin{eqnarray*}
|\{u\in Q_{n}\,:\, a_{u}\neq c_{u}\}| & \leq & |A\setminus B|+|\{u\in B\,:\, a_{u}\neq b_{u}\}|\;+\\
 &  & +\;|C\setminus B|+|\{u\in B\,:\, c_{u}\neq b_{u}\}|\end{eqnarray*}
But this is clear since, for each $u\in Q_{n}$ which contributes
to the left hand side, if $u\notin B$ then $u$ contributes twice
to the right hand side (to both $|A\setminus B|$ and $|C\setminus B|$),
and if $u\in B$ it contributes at least once (since $a_{u}\neq c_{u}$
implies we can have both $a_{u}=b_{u}$ and $c_{u}=b_{u}$). 

In the general case, let $B'=B\cap(A\cup C)$. Then the analysis above
shows that\begin{eqnarray*}
d(a,c) & \leq & \frac{1}{|A\cup B'|}\left(|A\setminus B'|+|\{u\in B'\,:\, a_{u}\neq b_{u}\}|\right)\;+\\
 &  & +\;\frac{1}{|C\cup B'|}\left(|C\setminus B'|+|\{u\in B'\,:\, c_{u}\neq b_{u}\}|\right)\end{eqnarray*}
The expression for $d(a,b)+d(b,c)$ is obtained from the right hand
side by adding $|B\setminus A|$ to both numerator and denominator
of the first term on the right, and similarly adding $|B\setminus C|$
to both the numerator and denominator of the second term. This has
the effect of increasing them, and we obtain the triangle inequality. 
\end{proof}

\subsection{\label{sub:properties}The invariant}

For the sake of completeness we present a slight generalization of
the invariants given in the introduction, aimed at accommodating other
growth scales. A \emph{growth function }$\rho:[0,\infty)^{\mathbb{N}}\rightarrow\mathbb{R}$
is a function of real-valued sequences $\underline{s}=s_{1},s_{2},\ldots$
satisfying
\begin{enumerate}
\item [(a)]Tail property: If $s_{i}=t_{i}$ for all sufficiently large
$i$, then $\rho(\underline{s})=\rho(\underline{t})$;
\item [(b)]Scaled monotonicity: If $s_{i}\leq ct_{i}$ for $c>0$ then
$\rho(\underline{s})\leq\rho(\underline{t})$. 
\end{enumerate}
The principle example we shall be interested in is \[
\rho_{slow}(\underline{s})=\limsup_{n\rightarrow\infty}\frac{\log(\log s_{n})}{\log n}\]
We shall also refer to \[
\rho_{exp}(\underline{s})=\limsup_{n\rightarrow\infty}\frac{\log s_{n}}{n}\]
Given a growth function $\rho(\cdot)$ and $\mathcal{P}$, $\varepsilon$
and $A$ as above, let\begin{equation}
\rho(T,\mu,A,\mathcal{P})=\lim_{\varepsilon\searrow0}\left(\limsup_{n\rightarrow\infty}\rho\left((N(A,d_{\mathcal{P},n},\varepsilon))_{n=1}^{\infty}\right)\right)\label{eq:rho-entropy-of-partition}\end{equation}
We shall show that this is independent of $A$ (Corollary \ref{cor:covering-rho-entropy-on-diffferent-sets}),
and denote it $\rho(T,\mu,\mathcal{P})$ for short. Finally, define
the $\rho$-entropy of the action by\[
\rho(T,\mu)=\sup\rho(T,\mu,\mathcal{P})\]
where the supremum is over co-finite partitions $\mathcal{P}$.

\subsection{\label{sub:Analysis}Analysis}

We begin the analysis with some elementary facts about covering numbers.
Given a pseud-metric $d'$ on $\Omega$ it will be convenient to define
the quantities $\diam(E,d')$ and $N(A,d',\varepsilon)$ in the same
way as in the introduction, where the definitions were given for $d'=d_{\mathcal{P},n}$.
Note that $N(A,d',\varepsilon)$ is non-decreasing as $\varepsilon\searrow0$. 
\begin{lem}
\label{lem:covering-numbers-in-comparable-metrics}Let $d^{1},d^{2}$
be pseudo-metrics on $\Omega$. Let $a\geq1$ and $\delta>0$, and
suppose that $A_{0}\subseteq A$ satisfies $\mu(A_{0})>(1-\delta)\mu(A)$
and $d^{2}(x,y)\leq ad^{1}(x,y)+\delta$ for $x,y\in A_{0}$. Then
$N(A,d^{2},a\varepsilon+\delta)\leq N(A,d^{1},\varepsilon)$.\end{lem}
\begin{proof}
Let $m=N(A,d^{1},\varepsilon)$ and let $E_{1},\ldots,E_{m}$ be an
optimal $(\varepsilon,d^{1})$-cover of $A$. Set $E'_{i}=E_{i}\cap A_{0}$,
so that \[
\mu(A\setminus\bigcup_{i=1}^{m}E'_{i})\leq\mu(A\setminus\bigcup_{i=1}^{m}E_{i})+\mu(A\setminus A_{0})\leq(\varepsilon+\delta)\mu(A)\leq(a\varepsilon+\delta)\mu(A)\]
Since $E'_{i}\subseteq A_{0}$, and using the inequality between $d^{2}$and
$d^{1}$, we have \[
\diam(E'_{i},d^{2})\leq a\diam(E'_{i},d^{1})+\delta\leq a\diam(E_{i},d^{1})+\delta<a\varepsilon+\delta\]
Therefore, $E'_{1},\ldots,E'_{m}$ is a $(d^{2},a\varepsilon+\delta)$-almost
cover of $A$, so $N(A,d^{2},a\varepsilon+\delta)\leq m=N(A,d^{1},\varepsilon)$.
\end{proof}
We now turn to the analysis of $\rho$-entropy. Recall that a partition
$\mathcal{R}$ refines a partition $\mathcal{P}$ if every atom of
$\mathcal{R}$ is a subset of an atom of $\mathcal{R}$.
\begin{lem}
\label{lem:covering-numbers-of-refining-partitions}If $\mathcal{R}$
refines $\mathcal{P}$ then $N(A,d_{\mathcal{P},n},\varepsilon)\leq N(A,d_{\mathcal{R},n},\varepsilon)$
and consequently $\rho(T,\mu,A,\mathcal{P})\leq\rho(T,\mu,A,\mathcal{R})$.\end{lem}
\begin{proof}
By the previous lemma it suffices to show that $d_{\mathcal{P},n}(x,y)\leq d_{\mathcal{R},n}(x,y)$
for all $x,y\in\Omega$. Consider the intermediate partition \[
\mathcal{S}=\{P_{1},R\,:\, R\in\mathcal{R}\mbox{ and }R\cap P_{1}=\emptyset\}\]
Note that $\mathcal{S}$ refines $\mathcal{P}$. We first claim that
\[
d_{\mathcal{P},n}(x,y)\leq d_{\mathcal{S},n}(x,y)\]
Indeed, in the definition of these quantities, the expressions in
the denominator agree for $\mathcal{P}$ and $\mathcal{S}$ since
they have the same core. On the other hand the numerator corresponding
to $\mathcal{P}$ is no greater than that of $\mathcal{S}$, because
any $u\in Q_{n}$ which satisfies $\mathcal{S}(T^{u}x)=\mathcal{S}(T^{u}y)$
then it certainly also satisfies $\mathcal{P}(T^{u}x)=\mathcal{P}(T^{u}y)$,
since $\mathcal{S}$ refines $\mathcal{P}$. It now remains to show
that\[
d_{\mathcal{S},n}(x,y)\leq d_{\mathcal{R},n}(x,y)\]
This follows from the fact that, when one compares the ratio defining
the left and right hand sides, one finds that each $u\in Q_{n}$ which
contributes to the numerator of the right hand side but not the left,
also contributes the same amount to the denominator of the right hand
side, but not the left. 

The last statement of the lemma is immediate from the definitions
and monotonicity of $N(A,d_{\mathcal{P},n},\varepsilon)$ in $\varepsilon$.
\end{proof}
Recall that for $F\subseteq\mathbb{Z}^{k}$, the $F$-refinement of
$\mathcal{P}$ is the partition $\mathcal{P}^{F}=\bigvee_{u\in F}T^{u}\mathcal{P}$,
where $T^{u}\mathcal{P}=\{T^{u}P\,:\, P\in\mathcal{P}\}$. This is
the coarsest partition which refines $T^{u}\mathcal{P}$ for all $u\in F$.
Note that $\mathcal{P}^{F}$ is co-finite if $\mathcal{P}$ is.
\begin{lem}
\label{lem:covering-number-of-P-to-the-F}If $0\in F\subseteq\mathbb{Z}^{k}$
is finite and $\mathcal{R}=\mathcal{P}^{F}$, then $N(A,d_{\mathcal{R},n},\varepsilon)\leq N(A,d_{\mathcal{P},n},\varepsilon/|F|)$,
and in particular $\rho(T,\mu,A,\mathcal{P})=\rho(T,\mu,A,\mathcal{R})$.\end{lem}
\begin{proof}
By Lemma \ref{lem:covering-numbers-in-comparable-metrics}, it suffices
to show that $d_{\mathcal{R},n}(x,y)\leq|F|\cdot d_{\mathcal{P},n}(x,y)$
for all $x,y\in\Omega$. Write \[
U=\{u\in Q_{n}\,:\, T^{u}x\notin P_{1}\mbox{ or }T^{u}y\notin P_{1}\}\]
for the union of the times in $Q_{n}$ at which $x$ or $y$ visit
the core of $\mathcal{P}$. Denote the set of $u\in Q_{n}$ where
the $(\mathcal{P},n)$-names differ by \[
D=\{u\in Q_{n}\,:\,\mathcal{P}(T^{u}x)\neq\mathcal{P}(T^{u}y)\}\]
By definition $d_{\mathcal{P},n}(x,y)=|D|/|U|$. Let $U',D'$ be defined
similarly using $\mathcal{R}$ instead of $\mathcal{P}$ (recall that
$R_{1}\in\mathcal{R}$ is the infinite atom of $\mathcal{R}$). Notice
that $D'=(D+Q_{n})\cap Q_{n}$ and $U'=(U+Q_{n})\cap Q_{n}$, where
$U+Q_{n}=\{u+v\,:\, u\in U\,,\, v\in Q_{n}\}$ and similarly for the
other expression. We clearly have $|D'|\leq|F|\cdot|D|$, and since
$0\in F$ we also have $|U'|\geq|U|$. Thus\[
d_{\mathcal{R},n}(x,y)=\frac{|D'|}{|U'|}\leq|F|\cdot\frac{|D|}{|U|}=|F|\cdot d_{\mathcal{P},n}(x,y)\]
as desired. 

For the second part of the lemma, notice that the first part implies
that $\rho(T,\mu,A,\mathcal{P})\geq\rho(T,\mu,A,\mathcal{R})$, while
the reverse inequality was proved in the previous lemma.
\end{proof}
For co-finite partitions $\mathcal{P}=\{P_{1},\ldots,P_{r}\}$ and
$\mathcal{R}=\{R_{1},\ldots,R_{r}\}$ of the same size, define the
distance \[
\Delta(\mathcal{P},\mathcal{R})=\frac{\mu(\bigcup_{i=1}^{r}(P_{i}\Delta R_{i}))}{\mu(\bigcup_{i=2}^{r}(P_{i}\cup Q_{i}))}\]
This is finite when $\mathcal{P},\mathcal{R}$ are co-finite (we continue
to assume that $P_{1},R_{1}$ are the infinite atoms). One can show
that this is a metric in a similar manner to the proof of Lemma \ref{lem:pseudo-metric-verification}
but we shall not need this.
\begin{lem}
\label{lem:covering-numbers-of-L1-close-partitions}If $\mathcal{P}=\{P_{1},\ldots,P_{r}\},\mathcal{R}=\{R_{1},\ldots,R_{r}\}$
are co-finite partitions and $\Delta(\mathcal{P},\mathcal{R})<\varepsilon$,
then $N(A,d_{\mathcal{P},n},6\varepsilon)\leq N(A,d_{\mathcal{R},n},\varepsilon)$
for sufficiently large $n$, and visa versa. In particular, $\rho(T,\mu,A,\mathcal{P})\leq\rho(T,\mu,A,\mathcal{R})$
and visa versa.\end{lem}
\begin{proof}
By the ratio ergodic theorem, there is an $n_{0}$ and a set $A'\subseteq A$
with $\mu(A')>(1-\varepsilon)\mu(A)$, such that for $n>n_{0}$ and
$x\in A'$, \[
\frac{|\{u\in Q_{n}\,:\, T^{u}x\in\bigcup_{i=1}^{r}(P_{i}\Delta R_{i})\}|}{|\{u\in Q_{n}\,:\, T^{u}x\in\bigcup_{i=2}^{r}(P_{i}\cup R_{i})\}||}<\varepsilon\]
which easily implies \[
d(x_{\mathcal{P},n},x_{\mathcal{R},n})\leq\frac{\varepsilon}{1-\varepsilon}<2\varepsilon\]
It follows that\[
d_{\mathcal{P},n}(x,y)\leq d(x_{\mathcal{P},n},x_{\mathcal{R},n})+d(x_{\mathcal{R},n},y_{\mathcal{R},n})+d(y_{\mathcal{R},n},y_{\mathcal{P},n})\leq d_{\mathcal{R},n}(x,y)+4\varepsilon\]
The inequality $N(A,d_{\mathcal{P},n},6\varepsilon)\leq N(A,d_{\mathcal{R},n},\varepsilon)$
now follows from Lemma \ref{lem:covering-numbers-in-comparable-metrics},
and the last claim by applying $\rho(\cdot)$ and using monotonicity
in $\varepsilon$.
\end{proof}
Write $\sigma(\mathcal{P})$ for the smallest sub-$\sigma$-algebra
in $\mathcal{B}$ with respect to which $\mathcal{P}$ is measurable.
\begin{lem}
\label{lem:covering-numbers-of-generating-partitions}If $\mathcal{P}_{n}$
is a refining sequence of co-finite partitions such that $\sigma(\bigvee_{n=1}^{\infty}\mathcal{P}_{n}^{\mathbb{Z}^{k}})=\mathcal{B}\bmod\mu$,
then $\sup_{\mathcal{R}}\rho(T,\mu,A,\mathcal{R})=\sup_{n}\rho(T,\mu,A,\mathcal{P}_{n})$.
In particular if $\mathcal{P}$ generates then $\sup_{\mathcal{R}}\rho(T,\mu,A,\mathcal{R})=\rho(T,\mu,A,\mathcal{P})$\end{lem}
\begin{proof}
Write $\beta=\sup_{\mathcal{R}}\rho(T,\mu,A,\mathcal{R})$. Fix $\varepsilon>0$
and a co-finite partition $\mathcal{R}$ such that $\rho(T,\mu,A,\mathcal{R})\geq\beta-\varepsilon$.
Also let $\delta>0$. Since $\rho(T,\mu,A,\mathcal{P}_{n})$ is non-decreasing
in $n$ it follows that for each atom $R_{i}\in\mathcal{R}$ and all
large enough $n$ there is an $r=r(n,i)$ and $P_{n,i}\in\sigma(\mathcal{P}_{n}^{Q_{r}})$
such that $\mu(R_{i}\Delta P_{n,i})<\delta$. It follows that we can
choose an $n$ and $r$ such that there is a partition $\mathcal{P}$
which is coarser than $\mathcal{P}_{n}^{Q_{r}}$, and such that $\Delta(\mathcal{P},\mathcal{R})<\delta$.
Hence \begin{eqnarray*}
\rho(T,\mu,A,\mathcal{P}_{n}) & = & \rho(T,\mu,A,\mathcal{P}_{n}^{Q_{r}})\\
 & \geq & \rho(T,\mu,A,\mathcal{P})\\
 & \geq & \rho(T,\mu,A,\mathcal{P},\delta)\\
 & \geq & \rho(T,\mu,A,\mathcal{R},6\delta)\end{eqnarray*}
where the equality is by Lemma \ref{lem:covering-number-of-P-to-the-F},
the first inequality is by Lemma \ref{lem:covering-numbers-of-refining-partitions},
the second equality is by definition, and the last inequality is by
Lemma \ref{lem:covering-numbers-of-L1-close-partitions}. Taking the
limit as $\delta\rightarrow0$ gives \[
\rho(T,\mu,A,\mathcal{P}_{n})\geq\rho(T,\mu,A,\mathcal{R})\geq\beta-\varepsilon\]
The claim follows.
\end{proof}

\begin{lem}
\label{lem:metric-and-shifted-metric}There is a set $\Omega_{0}\subseteq\Omega$
with $\mu(\Omega\setminus\Omega_{0})=0$ such that for every $0<\varepsilon<1$
and every $u\in\mathbb{Z}^{k}$, if $x,y\in\Omega$ then \begin{equation}
|d_{\mathcal{P},n}(x,y)-d_{\mathcal{P},n}(T^{u}x,T^{u}y)|<\varepsilon\label{eq:metric-bound-for-shifted-sets}\end{equation}
for all large enough $n$\end{lem}
\begin{proof}
Let $\Omega_{0}$ denote the set of points for which the Chacon-Ornstein
lemma (Theorem \ref{thm:Chacon-Ornstein}) holds for the function
$1_{\Omega\setminus P_{1}}$. We claim this is the desired set. Let
$x,y\in\Omega_{0}$, let\begin{eqnarray*}
U & (x)= & \{v\in\mathbb{Z}^{k}\,:\,\mathcal{P}(T^{v}x)=1\}\end{eqnarray*}
and similarly $U(y)$, and let $U=U(x)\cup U(y)$. By our choice of
$\Omega_{0}$, if $n$ is large enough then\begin{eqnarray*}
\frac{|U\cap(Q_{n}\Delta(Q_{n}+u))|}{|U\cap Q_{n}|} & \leq & \frac{|U(x)\cap(Q_{n}\Delta(Q_{n}+u)|}{|U(x)\cap Q_{n}|}+\frac{|U(y)\cap(Q_{n}\Delta(Q_{n}+u)|}{|U(y)\cap Q_{n}|}\\
 & < & \varepsilon\end{eqnarray*}
since two ratios on the right hand side are just the ratios in the
Chacon-Ornstein theorem applied to $1_{\Omega\setminus P_{1}}$ and
the points $x,y$. This implies that, in the expressions for $d_{\mathcal{P},n}(x,y)$
and $d_{\mathcal{P},n}(T^{u}x,T^{u}y)$, the numerators differ by
at most a multiplicative factor of $(1\pm\varepsilon)$, and similarly
the denominators. Using the fact that $(1+\varepsilon)/(1-\varepsilon)\leq1+4\varepsilon$
for $0<\varepsilon<1$ and the fact that $d_{\mathcal{P},n}\leq1$,
we obtain\[
|d_{\mathcal{P},n}(x,y)-d_{\mathcal{P},n}(T^{u}x,T^{u}y)|\leq4\varepsilon\]
the lemma follows.\end{proof}
\begin{cor}
\label{cor:covering-numbers-of-shifted-partitions}Let $\varepsilon>0$
and $u\in\mathbb{Z}^{k}$. Then $N(A,\mathcal{P},n,\varepsilon)\leq N(T^{u}A,\mathcal{P},n,\varepsilon/2)$
for all sufficiently large $n$.\end{cor}
\begin{proof}
Write $d_{\mathcal{P},n}^{u}(x,y)=d_{\mathcal{P},n}(T^{u}x,T^{u}y)$.
This is a pseudo-metric on $A$ and by the previous lemma we can find
$n_{0}\in\mathbb{N}$ and $A_{0}\subseteq A$ with $\mu(A_{0})>(1-\varepsilon)\mu(A)$
and such that \eqref{eq:metric-bound-for-shifted-sets} holds for
$n>n_{0}$ and $x,y\in A_{0}$. It follows from Lemma \ref{lem:covering-numbers-in-comparable-metrics}
that $N(A,d_{\mathcal{P},n},\varepsilon)\leq N(A,d_{\mathcal{P},n}^{u},\varepsilon/2)$,
and the result follows since $N(A,d_{\mathcal{P},n}^{u},\varepsilon/2)=N(T^{u}A,d_{\mathcal{P},n},\varepsilon/2)$.
\end{proof}

\begin{lem}
\label{lem:covering-number-on-different-sets}If $A\subseteq B$ are
sets of positive finite $\mu$-measure then, for all large enough
$n$, \[
N(A,\mathcal{P},n,c_{1}\varepsilon)\leq N(B,\mathcal{P},n,\varepsilon)\leq c_{2}N(A,\mathcal{P},n,\varepsilon/c_{3})\]
where $c_{1},c_{2},c_{3}$ are constants which do not depend on $n$.
In particular, $\rho(T,\mu,A,\mathcal{P})=\rho(T,\mu,B,\mathcal{P})$.\end{lem}
\begin{proof}
Fix $0<\varepsilon<1$ and let $E_{1},\ldots,E_{m}$ be a collection
of size $m=N(B,d_{\mathcal{P},n},\varepsilon)$ which is an $\varepsilon$-almost
cover of $B$. Then \[
\mu(A\setminus\bigcup_{i=1}^{m}E_{i})\leq\mu(B\setminus\bigcup_{i=1}^{m}E_{i})=\varepsilon\mu(B)=\varepsilon\frac{\mu(B)}{\mu(A)}\cdot\mu(A)\]
so $E_{1},\ldots,E_{m}$ is an $\varepsilon\mu(B)/\mu(A)$-almost
cover of $A$. It follows that $N(A,d_{\mathcal{P},n},\frac{\mu(B)}{\mu(A)}\varepsilon)\leq N(B,d_{\mathcal{P},n},\varepsilon)$.

For the other inequality we argue as follows. By ergodicity we have
$\mu(\Omega\setminus\bigcup_{u\in\mathbb{Z}^{k}}T^{u}A)=0$, so there
is a finite set $F\subseteq\mathbb{Z}^{k}$ such that $\mu(B\setminus\bigcup_{u\in F}T^{u}A)<\frac{\varepsilon}{2}\mu(B)$.
Write $\widetilde{A}=\bigcup_{u\in F}T^{u}A$, and note that $\mu(\widetilde{A})\leq|F|\cdot\mu(A)$,
so\[
\frac{\mu(B\cap\widetilde{A})}{\mu(\widetilde{A})}\geq\frac{(1-\varepsilon/2)\mu(B)}{|F|\mu(A)}\geq\frac{1}{2|F|}\]
By the previous corollary, for each $u\in F$ and large enough $n$
we have \[
N(T^{u}A,d_{\mathcal{P},n},\frac{\varepsilon}{4|F|})\leq N(A,d_{\mathcal{P},n},\frac{\varepsilon}{8|F|})\]
so for large enough $n$ we have \[
N(\widetilde{A},d_{\mathcal{P},n},\frac{\varepsilon}{4|F|})\leq|F|\cdot N(A,d_{\mathcal{P},n},\frac{\varepsilon}{8|F|})\]
Applying the first part of the current lemma to the containment $B\cap\widetilde{A}\subseteq\widetilde{A}$,
we find that\[
N(B\cap\widetilde{A},d_{\mathcal{P},n},\frac{\varepsilon}{2})\leq N(\widetilde{A},d_{\mathcal{P},n},\frac{\mu(B\cap\widetilde{A})}{\mu(\widetilde{A})}\cdot\frac{\varepsilon}{2})\leq N(\widetilde{A},d_{\mathcal{P},n},\frac{\varepsilon}{4|F|})\]
Since $\mu(B\cap\widetilde{A})\geq(1-\varepsilon/2)\mu(B)$ it follows
that \[
N(B,d_{\mathcal{P},n},\varepsilon)\leq N(B\cap\widetilde{A},d)(\frac{\varepsilon}{2})\]
combining the last three inequalities gives the desired result\end{proof}
\begin{cor}
\label{cor:covering-rho-entropy-on-diffferent-sets}If $A,B$ are
sets of finite measure then $\rho(T,\mu,\mathcal{P},A)=\rho(T,\mu,\mathcal{P},B)$.\end{cor}
\begin{proof}
Let $C=A\cup B$. Since $A\subseteq C$ and $B\subseteq C$, the previous
lemma and monotonicity of covering numbers in $\varepsilon$ implies
that $\rho(T,\mu,\mathcal{P},A)=\rho(T,\mu,\mathcal{P},C)$ and $\rho(T,\mu,\mathcal{P},B)=\rho(T,\mu,\mathcal{P},C)$,
and the conclusion follows.
\end{proof}

\subsection{\label{sub:Connections-with-recurrence}Connections with recurrence}

In the previous sections we defined $\rho$-entropy using name counts.
In this section we give a slightly simpler characterization in terms
of the complexity of recurrence patterns. 

Let $\rho(\cdot)$ be a growth function and $A\subseteq\Omega$ a
set of positive and finite measure. For $x\in A$, recall that the
pattern of returns to $A$ {}``up to time $n$'' is \[
R_{n}(A,x)=\{u\in\mathbb{Z}^{k}\,:\,\left\Vert u\right\Vert \leq n\mbox{ and }T^{u}x\in A\}\]
In order to compare return patterns of $x,y\in A$, for $n\in\mathbb{N}$
introduce the distance $d_{A,n}(x,y)$ by\[
d_{A,n}(x,y)=\frac{|R_{n}(A,x)\Delta R_{n}(A,y)|}{|R_{n}(A,x)\cup R_{n}(A,y)|}\]
Next, let $N_{n}(A,\varepsilon)=N(A,d_{A,n},\varepsilon)$ and set
\[
\widetilde{\rho}(T,\mu)=\sup_{A}\left(\lim_{\varepsilon\searrow0}\left(\limsup_{n\rightarrow\infty}\rho\left((N_{n}(A,\varepsilon))_{n=1}^{\infty}\right)\right)\right)\]
where the supremum is over measurable $A\subseteq\Omega$ of positive
finite measure.
\begin{prop}
\label{pro:recurrence-defintion-of-rho-entropy}$\widetilde{\rho}(T,\mu)=\rho(T,\mu)$
for i.m.p. actions and zero-entropy f.m.p. actions. $\widetilde{\rho}_{slow}(T,\mu)=\rho_{slow}(T,\mu)$
in all cases.\end{prop}
\begin{proof}
Note that $d_{A,n}(x,y)=d_{\{A,\Omega\setminus A\},n}(x,y)$, where
the right hand side is as in equation \eqref{eq:Pn-distance}. This
implies that $\widetilde{\rho}(T,\mu)\leq\rho(T,\mu)$ for all actions. 

For the reverse inequality, when the action is i.m.p. or f.m.p. with
entropy zero, choose $A$ such that $\mathcal{P}=\{A,\Omega\setminus A\}$
generates for the action (in the former case this can be done by Krengel's
theorem \cite{Krengel970}, in the latter by Krieger's generator theorem
\cite{Krieger1970}). By Lemma \ref{lem:covering-numbers-of-generating-partitions}
we have \[
\lim_{\varepsilon\searrow0}\left(\limsup_{n\rightarrow\infty}\rho\left((N_{n}(A,\varepsilon))_{n=1}^{\infty}\right)\right)=\rho(T,\mu,A,\{A,\Omega\setminus A\})=\rho(T,\mu)\]
and hence $\widetilde{\rho}\geq\rho$.

For $\rho_{slow,}$ it remains to prove equality for positive-entropy
f.m.p. actions. In this case observe that one can always find a set
$A$ such that the partition $\{A,\Omega\setminus A\}$ has positive
entropy; hence $\widetilde{\rho}_{slow}((N(A,d_{A,n},\varepsilon))_{n=1}^{\infty})=1$
for all small enough $\varepsilon>0$, and so $\widetilde{\rho}_{slow}(T,\mu)\geq1=\rho_{slow}(T,\mu)$.
The reverse inequality was established at the beginning of the proof. 
\end{proof}
Recall from the introduction that \[
\alpha(T,\mu)=\limsup_{n\rightarrow\infty}\frac{\log|R_{n}(A,x)|}{\log|Q_{n}|}\]
which is independent of $A$ and a.s. independent of $x$. We show
next that $\rho_{slow}(T,\mu)\leq\alpha(T,\mu)$.
\begin{proof}
[Proof of Proposition \ref{pro:recurrence-bound-for-rho-slow}] Using
the trivial binomial bound $\binom{n}{m}\leq n^{m}$, for $\delta>0$
we find that the number of subsets $E\subseteq Q_{n}$ with $|E|\leq|Q_{n}|^{\delta}$
is at most \[
\binom{|Q_{n}|}{\left\lceil |Q_{n}|^{\delta}\right\rceil }\leq|Q_{n}|^{\left\lceil |Q_{n}|^{\delta}\right\rceil }\leq|Q_{n}|^{2|Q_{n}|^{\delta}}=2^{2|Q_{n}|^{\delta}\log|Q_{n}|}\]

Given a set $A\subseteq\Omega$ of finite measure and $\varepsilon\geq0$,
by definition of $\alpha=\alpha(T,\mu)$ there is an $n_{0}$ and
a subset $A_{\varepsilon}\subseteq A$ such that $\mu(A_{\varepsilon})>(1-\varepsilon)\mu(A)$,
and if $n>n_{0}$ and $x\in A_{\varepsilon}$ then $|R_{n}(A,x)|\leq|Q_{n}|^{\alpha+\varepsilon}$.
Fix $n>n_{0}$. For each $E\subseteq Q_{n}$ of size $|E|\leq|Q_{n}|^{\alpha+\varepsilon}$
let $A_{\varepsilon}^{E}=\{x\in A_{\varepsilon}\,:\, R_{n}(A,x)=E\}$.
Clearly $\diam(A_{\varepsilon}^{E},d_{A,n})=0$. On the other hand
the estimate above tells us that one can cover $A_{\varepsilon}$
by $2^{2|Q_{n}|^{\alpha}\log|Q_{n}|}$ sets of this form. Therefore,
for $n>n_{0}$ we have \[
N(A,d_{A,n},\varepsilon)\leq2^{2|Q_{n}|^{\alpha+\varepsilon}\log|Q_{n}|}\]
Plugging this into the definition we find that $\widetilde{\rho}(T,\mu)\leq\alpha+\varepsilon$,
and hence by the previous proposition the same is true for $\rho_{slow}(T,\mu)$.
Since $\varepsilon>0$ was arbitrary, this proves the claim.
\end{proof}

\section{\label{sec:Lipschitz-actions}Lipschitz actions and examples}

\subsection{\label{sub:Lipshifts-actions-dimension-and-entropy}Lipschitz actions,
dimension and slow entropy }

A map $f:\Omega\rightarrow\Omega$ is Lipschitz if there is a constant
$C$ such that $d(fx,fy)\leq Cd(x,y)$. The smallest constant with
this property is denoted $\lip(f)$ . Note that\[
\lip(f\circ g)\leq\lip(f)\cdot\lip(g)\]
Note also that diffeomorphisms of Riemannian manifolds are Lipschitz
maps. 
\begin{lem}
\label{lem:lipschitz-constant-bound}Suppose $\mathbb{Z}^{k}$ acts
on a metric space $(\Omega,d)$ by Lipschitz maps. Then there is a
constant $C$ such that $\lip T^{u}\leq C^{\left\Vert u\right\Vert }$.\end{lem}
\begin{proof}
Let $T_{1},\ldots,T_{k}$ be generators of the action and \[
C=\sup_{i=1,\ldots,k}\{\lip(T_{i}),\lip(T_{i}^{-1})\}\]
Then for $u=(u_{1},\ldots,u_{k})\in\mathbb{Z}^{k}$,\begin{eqnarray*}
\lip(T^{u}) & = & \lip(T^{u_{1}}T^{u_{2}}\ldots T^{u_{n}})\\
 & \leq & \prod_{i=1}^{k}\lip(T^{u_{i}})\\
 & \leq & \prod_{i=1}^{k}(\lip T_{i})^{|u_{i}|}\\
 & \leq & C^{\sum|u_{i}|}\end{eqnarray*}
the claim follows.
\end{proof}
Recall that if $(\Omega,d)$ is a compact metric space then the $\varepsilon$-separation
number $\sep(\Omega,d,\varepsilon)$ is the size of the largest $E\subseteq\Omega$
such that $d(x,y)\geq\varepsilon$ for distinct $x,y\in E$. The upper
Minkowski (box) dimension $(\Omega,d)$ is\[
\bdim(\Omega,d)=\limsup_{\varepsilon\searrow0}\frac{\log\sep(\Omega,d,\varepsilon)}{\log\varepsilon}\]
so $\bdim(\Omega,d)<\alpha$ implies that $\sep(\Omega,d,\varepsilon)\leq c\cdot(1/\varepsilon)^{\alpha}$.
Note that if $\Omega$ is a compact manifold of dimension $m$ and
$d$ is a Riemannian metric, then $\bdim(\Omega,d)=m$.

Finally, recall that when $\mathbb{Z}^{k}$ acts on a metric space
$(\Omega,d)$ then the Bowen metric on $\Omega$ is \begin{equation}
d_{n}^{\infty}(x,y)=\sup_{\left\Vert u\right\Vert \leq n}d(T^{u}x,T^{u}y)\label{eq:bowen-metric}\end{equation}

\begin{lem}
\label{lem:bound-on-bowen-separation}Suppose that $(\Omega,d)$ has
finite box dimension and that $\mathbb{Z}^{k}$ acts on it by Lipschitz
maps. Then there is a constant $C$ such that $\sep(\Omega,d_{n},\varepsilon)\leq C^{n}/\varepsilon^{C}$.\end{lem}
\begin{proof}
Let $C_{1}$ be a constant such that $\lip(T^{u})\leq C_{1}^{\left\Vert u\right\Vert }$,
and let $C_{2}$ be a constant such that $\sep(\Omega,d,\varepsilon)\leq C_{2}(1/\varepsilon)^{C_{2}}$.
Notice that if $d_{n}^{\infty}(x,y)\geq\varepsilon$ then there is
some $u$ with $\left\Vert u\right\Vert \leq n$ such that $d(T^{u}x,T^{u}y)\geq\varepsilon$,
so \[
d(x,y)\geq\frac{\varepsilon}{C_{1}^{\left\Vert u\right\Vert }}\geq\frac{\varepsilon}{C_{1}^{n}}\]
Therefore if $E\subseteq\Omega$ is a $(d_{n}^{\infty},\varepsilon)$-separated
set it is also $(d,\varepsilon/C_{1}^{n})$-separated, hence \[
\sep(\Omega,d_{n}^{\infty},\varepsilon)\leq\sep(\Omega,d,\varepsilon/C_{1}^{n})\leq C_{2}\cdot C_{1}^{C_{2}n}/\varepsilon^{C_{2}}\]
the claim follows.
\end{proof}
We now turn to the proof of Theorem \ref{thm:large-s-implies-no-diff-model}. 
\begin{lem}
\label{lem:slow-core-for-non-pure-actions}If $(\Omega,\mathcal{B},\mu,T)$
is an ergodic with non-uniform slow entropy and $\rho_{slow}(T,\mu)>\beta$,
then there is a co-finite partition $\mathcal{P}$ with core $A$
such that $\rho_{slow}(T,\mu,\mathcal{P})>\beta$ and $\rho_{slow}(T,\mu,\mathcal{P})>\rho_{slow}(T,\mu,\{\Omega\setminus A,A\})$.\end{lem}
\begin{proof}
Since the action has non-uniform slow entropy, there is a co-finite
partition $\mathcal{P}'=\{P'_{1},\ldots,P'_{m}\}$ such that $\rho_{slow}(T,\mu,\mathcal{P}')<\rho_{slow}(T,\mu)$.
Let $A$ denote the core of $\mathcal{P}'$ and let $\mathcal{P}''=\{\Omega\setminus A,A\}$.
Since $\mathcal{P}'$ is a refinement of $\mathcal{P}''$, by Lemma
\ref{lem:covering-numbers-of-refining-partitions} we have $\rho_{slow}(T,\mu,\mathcal{P}'')<\rho_{slow}(T,\mu)$.

Let $\mathcal{P}_{n}$ be a refining sequence of co-finite partitions
with common core $A$ and which separate points in $A$. By ergodicity,
the hypothesis of Lemma \ref{lem:covering-numbers-of-generating-partitions}
is satisfied, and we can choose an $n$ such that the partition $\mathcal{P}'''=\mathcal{P}_{n}$
satisfies $\rho_{slow}(T,\mu,\mathcal{P}''')>\rho_{slow}(T,\mu,\mathcal{P}'')$
and also $\rho_{slow}(T,\mu,\mathcal{P}''')>\beta$. This is the desired
partition.
\end{proof}
We now begin the proof of Theorem \ref{thm:large-s-implies-no-diff-model}.
Suppose that $(\Omega,\mathcal{B},\mu,T)$ is an ergodic $\mathbb{Z}^{k}$-action
by Lipschitz maps with respect to the compact, separable metric $d$
on $\Omega$. Assuming $\rho_{slow}(T,\mu)>1$ and the action does
not have uniform slow entropy, we will show that $\bdim(\Omega,d)=\infty$.

Let $\mathcal{P}=\{P_{1},\ldots,P_{m}\}$ be a co-finite partition
with core $A$ and write $\mathcal{A}=\{\Omega\setminus A,A\}$. Assume
that $\rho_{slow}(T,\mu,\mathcal{P})>1$ and $\rho_{slow}(T,\mu,\mathcal{P})>\rho_{slow}(T,\mu,\mathcal{A})$,
as we may by the previous lemma. Choose $\beta,\gamma$ and $\varepsilon>0$
such that $\gamma>1$ and\[
\rho_{slow}(A,d_{\mathcal{P},n},\varepsilon)>\gamma>\beta>\rho_{slow}(A,d_{\mathcal{A},n},\varepsilon)\]

Recall that a finite Borel measure $\mu$ on a complete separable
metric space is inner regular, i.e. \[
\mu(C)=\sup\{\mu(K)\,:\, K\subseteq C\mbox{ and }K\mbox{ is compact}\}\]
Applying this to the restriction of $\mu$ to $A$, define a partition
$\mathcal{R}$ as follows. Let $R_{1}=P_{1}$, so $\mathcal{R}$ and
$\mathcal{P}$ have the same core $A$, and replace each finite-measure
atom $P_{i}\in\mathcal{P}$ by sets $K_{i}$ and $P_{i}\setminus K_{i}$,
where $K_{i}$ is a compact set satisfying \begin{equation}
\mu(K_{i})>(1-\frac{\varepsilon}{8|\mathcal{P}|})\mu(P_{i})\label{eq:Ki-exhaust-Pi}\end{equation}
Note that $\mathcal{R}$ refines $\mathcal{P}$, so \[
\rho_{slow}(T,\mu,\mathcal{R},\varepsilon)>\gamma\]
Since the $K_{i}$ are compact there is a $\tau>0$ such that \[
\min\{d(x,y)\,:\, x\in K_{i}\;,\; y\in K_{j}\}>\tau\qquad\mbox{ for all }i\neq j\]
Also, write\[
K=\bigcup_{i=1}^{m}K_{i}\]

By the ratio ergodic theorem, we may choose an $n_{0}$ and a set
$A_{0}\subseteq A$ of measure $\mu(A_{0})>(1-\frac{\varepsilon}{4})\mu(A)$
such that, for $n>n_{0}$ and $x\in A_{0}$,\begin{equation}
\frac{\sum_{u\in Q_{n}}1_{P_{i}\setminus K_{i}}(T^{u}x)}{\sum_{u\in Q_{n}}1_{A}(T^{u}x)}<\frac{\mu(P_{i}\setminus K_{i})}{\mu(A)}+\frac{\varepsilon}{8|\mathcal{P}|}\qquad i=2,\ldots,m\label{eq:ration-ergodic-theorem-for-Ki}\end{equation}

\begin{lem}
\label{lem:from-pattern-to-bowen-metrics}Let $n>n_{0}$ and $x,y\in A_{0}$,
and suppose that $d_{\mathcal{A},n}(x,y)<\varepsilon/4$ and $d_{\mathcal{R},n}(x,y)\geq\varepsilon/2$.
Then $d_{n}^{\infty}(x,y)\geq\tau$.\end{lem}
\begin{proof}
Let $U=\{u\in Q_{n}\,:\, T^{u}x\in A\mbox{ or }T^{u}y\in A\}$. We
then have\begin{eqnarray*}
d_{\mathcal{R},n}(x,y) & = & \frac{\{u\in U\,:\, T^{u}x\notin A\mbox{ or}T^{u}y\notin A\}}{|U|}\;+\\
 &  & \;+\frac{\{u\in Q_{n}\,:\,\mathcal{R}(T^{u}x)\neq\mathcal{R}(T^{u}y)\mbox{ and }T^{u}x,T^{u}y\in A\}}{|U|}\end{eqnarray*}
The first term is just $d_{\mathcal{A},n}(x,y)$, so\[
\frac{\{u\in Q_{n}\,:\,\mathcal{R}(T^{u}x)\neq\mathcal{R}(T^{u}y)\mbox{ and }T^{u}x,T^{u}y\in A\}}{|U|}>d_{\mathcal{R},n}(x,y)-\frac{\varepsilon}{4}\geq\frac{\varepsilon}{4}\]
Now the set in the numerator of the left hand side can be written
as a sum of two terms, the first consisting of those $u$ for which
$T^{u}x\in A\setminus K$ or $T^{u}y\in A\setminus K$, and the second
of those for which both $T^{u}x,T^{u}y\in K$. Using \eqref{eq:ration-ergodic-theorem-for-Ki},
we see that the first term contributes at most $\varepsilon/4$. It
follows that\[
\frac{|\{u\in Q_{n}\,:\,\mathcal{R}(T^{u}x)\neq\mathcal{R}(T^{u}y)\mbox{ and }T^{u}x,T^{u}y\in K\}|}{|U|}>0\]
Thus there is at least one $u\in Q_{n}$ such that $T^{u}x\in K_{i}$
and $T^{u}y\in K_{j}$ for some $i\neq j$. Hence $d(T^{u}x,T^{u}y)>\tau$,
and so $d_{n}^{\infty}(x,y)>\tau$.
\end{proof}
Fix $n$ and let $E_{1},\ldots,E_{s}$ be a minimal $(d_{\mathcal{A},n},\varepsilon/4)$-almost
cover of $A$, i.e. $\diam(E_{i},d_{\mathcal{A},n})<\varepsilon/4$
and $\mu(A\setminus\bigcup_{i=1}^{s}E_{i})>(1-\varepsilon/4)\mu(A)$.
Since $\rho_{slow}(A,d_{\mathcal{A},n},\varepsilon/4)<\beta$, for
large enough $n$ we will have $s<2^{n^{\beta}}$.

Let $E'_{i}=E_{i}\cap A_{0}$. For each $i=1,\ldots,s$, if $\sep(E'_{i},d_{\mathcal{R},n},\varepsilon/2)<2^{n^{(1+\gamma)/2}}$
then we can cover $E'_{i}$ with $2^{n^{(1+\gamma)/2}}$ $d_{\mathcal{R},n}$-balls
of radius $\varepsilon/2$. If this held for all $i=1,\ldots,s$ we
would have a collection of at most $2^{n^{\beta}}\cdot2^{n^{(1+\gamma)/2}}$
sets of $d_{\mathcal{R},n}$-diameter $\varepsilon$ which completely
cover $\bigcup_{i=1}^{s}E'_{i}$, and since \[
\mu(A\setminus\bigcup_{i=1}^{s}E'_{i})\leq\mu(A\setminus A_{0})+\mu(A\setminus\bigcup_{i=1}^{s}E_{i})\leq\varepsilon/2\]
this would be a $(d_{\mathcal{R},n},\varepsilon)$-cover of $A$.
Since $2^{n^{\beta}}\cdot2^{n^{(1+\gamma)/2}}<2^{n^{\gamma}}$ for
all large enough $n$ and since $\rho_{slow}(A,d_{\mathcal{R},n},\varepsilon)>\gamma$,
there must be arbitrarily large $n$ for which the above fails, i.e.
there is a set $E\subseteq A_{0}$ of $d_{\mathcal{A},n}$-diameter
$\leq\varepsilon/4$ and $\sep(E,d_{\mathcal{R},n},\varepsilon/2)\geq2^{n^{(1+\gamma)/2}}$.
The last condition means that there is a collection $I\subseteq E$
of size $|I|\geq2^{n^{(1+\gamma)/2}}$ such that $d_{\mathcal{R},n}(x,y)\geq\varepsilon/2$
for distinct $x,y\in I$. Applying Lemma \ref{lem:from-pattern-to-bowen-metrics},
we have $d_{n}^{\infty}(x,y)\geq\tau$ for $x,y\in I$. Thus $\sep(\Omega,d_{n},\tau)\geq\sep(I,d_{n},\tau)>2^{n^{\gamma}}$
, and since this holds for infinitely many $n$ and since $(1+\gamma)/2>1$,
we conclude from Lemma \ref{lem:lipschitz-constant-bound} that $\bdim(\Omega,d)=\infty$.

This concludes the proof of Theorem \ref{thm:large-s-implies-no-diff-model}.

As a final remark, note that the proof above would have been much
simplified if we could find a compact set $K_{1}$ such that $\mu(P_{1}\setminus K_{1})$
is very close to $0$. However, since $\mu(P_{1})=\infty$, which
may not happen; that is, {}``most'' of the mass of $P_{1}$ may
accumulate near the core. The assumption of non-uniform slow entropy
in Theorem \ref{thm:large-s-implies-no-diff-model} allows us to avoid
this problem by {}``relativising'' the problem to $A$.

\subsection{\label{sub:Cutting-and-stacking}Cutting and stacking}

Our next task is to construct a $\mathbb{Z}^{2}$ action which does
not have uniform slow entropy and with $\rho_{slow}>1$. We shall
combine several methods. The first is cutting and stacking, though
a better name in the multidimensional case might be cutting and tiling.
We shall only require the rank-1 version, which we describe next.

Begin with an abstract atomless $\sigma$-finite Lebesgue space of
infinite measure, which we refer to as the \emph{pool}. We shall define
an action on increasing subsets of the pool, eventually arriving at
the desired action on a part or all of the pool. It is convenient
also to imagine that, as we define our subset and action, we also
color the points we use, thus defining a partition of the pool according
to the colors. We shall use this informal coloring procedure, since
a more formal definition does not seem to add much. 

Recall that any measurable subset $A$ of a Lebesgue space with measure
$w>0$ can be identified, in a measurable and measure-preserving fashion,
with an interval of length $w$ in the real line. An \emph{arrangement
of radius $r$ and width $w\subseteq\mathbb{R}^{2}$ }is an assignment
$u\mapsto A_{u}$ from $u\in Q_{r}$ to pairwise disjoint subsets
$A_{u}$ of the pool. The $A_{u}$ have measure $w$ and we identify
them with intervals of length $w$. Each $A_{u}$ is colored monochromatically
by a color which may depend on $u$. If $x\in A_{u}$ we call $u$
the position of $x$, and we say that $x$ is in the arrangement if
$x\in\bigcup_{u\in Q_{r}}A_{u}$. 

Given an arrangement as above there is a partial measure-preserving
action $T$ on $\bigcup_{u\in Q_{r}}A_{u}$, defined as follows. If
$x\in A_{u}$ and $v\in\mathbb{Z}^{2}$ is such that $u+v\in Q_{r}$,
then $T^{v}x\in A_{u+v}$ is the point corresponding to $x$ when
the intervals $A_{u}$ and $A_{u+v}$ are identified by a translation
(this is the reason we identify $A_{u}$ with intervals: it provides
canonical isomorphisms between them). 

The cutting and tiling construction is a recursive procedure in which,
at the $i$-th stage, one has an arrangement $\mathcal{A}_{i}=(A_{i,u})_{u\in Q_{r(i)}}$
of radius $r(i)$ and width $w(i)$ in which each $A_{i,u}$ is monochromatic,
along with the associated partial actions. One begins with some collection
of arrangements of radius $0$ (recall that $|Q_{0}|=1$). Assuming
we have carried out the construction up to step $i$, one constructs
$\mathcal{A}_{i+1}$ using the following two steps:
\begin{description}
\item [{Cutting}] Choose an integer $m$ and partition each interval $A_{i,u}$
into $m$ equal subintervals $A_{i,u,j}$, $1\leq j\leq m$, each
of length $w(i+1)=w(i)/m$. For each such $j$ we obtain a new arrangement
$\mathcal{A}_{i,j}=(A_{i,u,j})_{u\in Q_{r(i)}}$ of radius $r(i)$
and width $w(i+1)$.
\item [{Tiling}] Choose $r(i+1)$ and a function $\psi:\{1,\ldots,m\}\rightarrow Q_{r(i+1)-r(i)}$
such that $\left\Vert \psi(j_{1})-\psi(j_{2})\right\Vert \geq2r(i)$
for all $1\leq j_{1}<j_{2}\leq m$. For each $1\leq j\leq m$, translate
$\mathcal{A}_{i,j}$ to $\psi(j)$, thus defining $\mathcal{A}_{i+1}$
at the sites $\varphi(j)+u$, $u\in Q_{r(i)}$. More precisely, for
$u\in Q_{r(i)}$ set $A_{i+1,\psi(j)+u}=A_{i,u,j}$. Note that $\varphi(j)+u\in Q_{r(i+1)}$,
and no conflicts occur, due to our assumptions on $\psi$. Finally,
to the remaining $u\in Q_{r(i+1)}$ assign new intervals taken from
the pool, and assign a color to each interval.
\end{description}
Let $\Omega_{i}=\bigcup_{u\in Q_{r(i)}}A_{i,u}$ and $\Omega=\bigcup_{i=1}^{\infty}\Omega_{i}$,
and let $\mathcal{P}$ denote the measurable partition of $\Omega$
into monochromatic sets. Note that the partial action defined on $\Omega_{i+1}$
extends the partial action on $\Omega_{i}$, so on $\Omega$ we have
a well-defined measure-preserving partial action $T$. We would like
conditions which ensure that $T$ is an a.e.-defined action. To this
end, let $\varepsilon_{i}$ denote the total mass of points in $\Omega_{i}$
which in in $\mathcal{A}_{i+1}$ are located within distance $i$
of the boundary of $Q_{r(i+1)}$. If $\sum\varepsilon_{i}<\infty$,
then by Borel-Cantelli lemma, for every $v\in\mathbb{Z}^{k}$, a.e.
point $x\in\Omega$ belongs at some stage to an arrangement in which
is located at distance at least $\left\Vert v\right\Vert $ from the
boundary, and hence $T^{v}x$ is defined. Thus $T$ is a $\mathbb{Z}^{2}$-action
defined a.e. on $\Omega$. 

The following fact is standard:
\begin{prop}
\label{pro:ergodicity-of-rank-1-actions}Rank-one cutting and stacking
constructions produce ergodic actions.
\end{prop}

\subsection{\label{sub:a-rank-1-action}A rank-1 construction}

As a preliminary step in the proof of Theorem \ref{thm:non-realizable actions-exist}
we construct a certain action $(\Omega_{0},\mathcal{B}_{0},\mu_{0},T_{0})$
in the manner described above. We use only two colors, $0$ and $1$.
The construction will be determined by the sequence $r(i)\rightarrow\infty$,
upon which we place a few requirements. Write\[
s(i)=r(i+1)-r(i)\]
and\[
m(i)=s(i)^{1/3}\]
We assume $1\leq r(1)<r(2)<\ldots$ grows quickly enough that 
\begin{enumerate}
\item $m(i)$ is an integer and $m(i)>2r(i)$,
\item $\prod_{j\leq i}r(j)=r(i)^{1+o(1)}$
\end{enumerate}
Begin at step $i=1$ with a single arrangement of radius $r(0)=0$,
width $1$, and color $1$. At later steps we will only add mass colored
$0$.

For the inductive step, suppose we have already constructed the arrangement
$\mathcal{A}_{i}=(A_{i,u})_{u\in Q_{r(i)}}$ of width $w(i)$. Define
\[
\Gamma_{i}=Q_{s(i)}\cap m(i)\mathbb{Z}^{2}\]
Note that \[
|\Gamma_{i}|=(2\frac{s(i)}{m(i)}+1)^{2}=r(i+1)^{3/2+o(1)}\]
Cut $\mathcal{A}_{i}$ into $|\Gamma_{i}|$ sub-arrangements of equal
width $w(i+1)=w(i)/|\Gamma_{i}|$, which we denote $\mathcal{A}_{i,u}$,
$u\in\Gamma_{i}$. Let $\psi_{i}:\{1,\ldots,|\Gamma_{i}|\}\rightarrow\Gamma_{i}$
be a fixed bijection, and form $\mathcal{A}_{i+1}$ from this data
as described in the previous section, coloring all new intervals with
the color $0$. Note that $\psi_{i}$ satisfies the requirements in
the tiling stage by property (1) of the growth of $r(i)$.

Let $(\Omega_{0},\mathcal{B}_{0},\mu_{0},T_{0})$ denote the resulting
action, which is ergodic by Proposition \ref{pro:ergodicity-of-rank-1-actions}.
Let $A\subseteq\Omega_{0}$ denote the set of points colored $1$,
and note that $\mu(A)=1$ since we did not add any more points with
this color after the first step. 

It will be useful to have a more direct description of the coloring
of $Q_{r(i)}$ derived from $\mathcal{A}_{i}$. For sets $U,V\subseteq\mathbb{Z}^{k}$
define $U+V=\{u+v\,:\, u\in U\,,\, v\in V\}$, and similarly for more
than two summands. Also, abbreviate $U+v=U+\{v\}$. Let \[
\Gamma_{i}^{*}=\Gamma_{1}+\ldots+\Gamma_{i}\]
Then $u\in Q_{r(i)}$ is colored $1$ if and only if \[
u\in\Gamma_{1}+\ldots+\Gamma_{i-1}=\Gamma_{i-1}^{*}\]
Note that, because of the super-exponential growth of $m(i)$, such
a $u$ has a unique representation as $u=\sum_{j=1}^{i}u_{j}$ with
$u_{j}\in\Gamma_{j}$. We also have estimates on $|\Gamma_{i}^{*}|$:
\begin{claim}
$|\Gamma_{i}^{*}|=r(i-1)^{3/2+o(1)}$\end{claim}
\begin{proof}
On the one hand, \[
|\Gamma_{1}+\ldots+\Gamma_{i-1}|\geq|\Gamma_{i-1}|=r(i)^{3/2+o(1)}\]
On the other hand,\[
|\Gamma_{1}+\ldots+\Gamma_{i-1}|\leq\prod_{j=1}^{i-1}|\Gamma_{j}|=\prod_{j=1}^{i-1}r(j+1)^{3/2+o(1)}=r(i)^{3/2+o(1)}\]
where the last equality is by property (2) of the growth of $r(i)$.
Combining these gives the claim.\end{proof}
\begin{claim}
$\mu_{0}(\Omega_{0})=\infty$.\end{claim}
\begin{proof}
The total number of intervals in $\mathcal{A}_{i}$ is $|Q_{r(i)}|=r(i)^{2+o(1)}$.
The number of intervals in $\mathcal{A}_{i}$ which already appeared
in $\mathcal{A}_{i-1}$ is $|\Gamma_{i}^{*}|=r(i)^{3/2+o(1)}$. Since
the ratio tends to $\infty$ as $i\rightarrow\infty$, we have $\mu_{0}(\Omega_{i}\setminus\Omega_{i-1})\rightarrow\infty$,
so $\mu_{0}(\Omega)=\infty$.
\end{proof}
We do not require the next claim, since without it we can simply pass
to the factor of $(\Omega_{0},\mathcal{B}_{0},\mu_{0},T_{0})$ corresponding
to the smallest invariant $\sigma$-algebra containing $A$. Therefore
we only briefly indicate the proof.
\begin{claim}
\label{cla:A-generates}The partition $\mathcal{P}=\{\Omega_{0}\setminus A,A\}$
generates for $T_{0}$.\end{claim}
\begin{proof}
We claim that if $u\in Q_{r(i)}$ and $x\in A_{i,u}$ then $u$ can
be recovered from the $(\mathcal{P},2r(i))$-name of $x$. Indeed,
let it is easy to see that \[
R_{2r(i)}(A,x)=\Gamma_{i}^{*}-u\]
Since $\Gamma_{i}^{*}$ is symmetric about the axes, we can recover
$u$ from $R_{2r(i)}(A,x)$ by \[
u=-\frac{1}{|R_{2r(i)}(A,x)|}\sum_{v\in R_{2r(i)}(A,x)}v\]
It follows that $\bigvee_{u\in Q_{2r(i)}}T^{u}\mathcal{P}$ partitions
$\Omega_{i}$ into intervals of length $w(i)$, and since $w(i)\rightarrow0$,
we find that $\bigvee_{u\in\mathbb{Z}^{2}}T^{u}\mathcal{P}$ separates
points.\end{proof}
\begin{claim}
\label{cla:zero-slow-entropy}$\rho_{slow}(T_{0},\mu_{0})=0$.\end{claim}
\begin{proof}
Let us estimate $N(A,d_{A,r},\varepsilon)$ (see section \ref{sub:Connections-with-recurrence}).
Fix $r$ in the range \[
m(i-1)-2r(i-1)\leq n<m(i)-2r(i)\]
If $x\in A$ then $R_{n}(A,x)$ is completely determined by the position
of $x$ in $\mathcal{A}_{i}$, since by the construction distances
between occurrences of sub-arrangements of $\mathcal{A}_{i}$ in $\mathcal{A}_{j}$
for any $j\geq i$ are at least $m(i)-2r(i)$, which is greater than
$n$). As we already observed, the possible positions of $x$ in $\mathcal{A}_{i}$
are the vectors $u\in\Gamma_{i}^{*}$. Hence \begin{eqnarray*}
N(A,d_{A,n},\varepsilon) & \leq & |\Gamma_{i}^{*}|\\
 & \leq & r(i)^{3/2+o(1)}\\
 & \leq & m(i-1)^{9/2+o(1)}\\
 & \leq & (m(i-1)-2r(i-1))^{9/2+o(1)}\\
 & \leq & n^{9/2+o(1)}\\
 & \leq & 2^{o(1)n}\end{eqnarray*}
so $\rho_{slow}((N(A,d_{N,r},\varepsilon))_{r=1}^{\infty})=0$ for
every $\varepsilon>0$. Since $\{\Omega_{0}\setminus A,A\}$ generates,
it follows that $\rho_{slow}(T_{0},\mu_{0})=0$.
\end{proof}
Essentially the same computation gives:
\begin{claim}
\label{cla:recurrence-rate-bound}If $x\in\Omega_{i}$ then $|R_{2r(i)}(A,x)|=|Q_{2r(i)}|^{3/4+o(1)}=r(i)^{3/2+o(1)}$.
\end{claim}
In particular this implies that $\alpha(T,\mu_{0})\geq3/4$ (equality
also holds but we shall not need this). Combined with the previous
claim, this shows that there exist $\mathbb{Z}^{k}$-actions such
that $\rho(T,\mu)\leq k\alpha(T,\mu)$.

\subsection{\label{sub:Symbolic-construction}Symbolic systems}

It will be useful now to introduce the language of symbolic representation.
Let $\Sigma$ be a finite set considered as a discrete topological
space. The full $\mathbb{Z}^{2}$-shift over $\Sigma$ is the product
space $\Sigma^{\mathbb{Z}^{2}}$, endowed with the product topology
and the Borel $\sigma$-algebra, which we suppress in our notation.
The shift action $S$ is the continuous $\mathbb{Z}^{2}$ actions
given by \[
(S^{u}x)_{v}=x_{u+v}\]
We use the same symbol $S$ to represent the shift action on full
shifts over different alphabets.

Given an integer $m$ and a map $\pi:\Sigma^{Q_{m}}\rightarrow\Delta$
to another finite set $\Delta$ of symbols, there is an induced map
$\overline{\pi}:\Sigma^{\mathbb{Z}^{2}}\rightarrow\Delta^{\mathbb{Z}^{2}}$
defined by \[
(\overline{\pi}x)_{u}=\pi((S^{u}x)|_{Q_{m}})\]
This is a factor map, i.e. $S^{u}\overline{\pi}=\overline{\pi}S^{u}$.
Given an invariant measure $\nu$ on $\Sigma^{\mathbb{Z}^{2}}$ the
push-forward $\overline{\nu}=\overline{\pi}\nu$ of $\nu$ is $S$
invariant, and $\overline{\pi}$ is then a factor map between the
measure preserving systems $(\Sigma^{\mathbb{Z}^{2}},\nu,S)$ and
$(\Delta^{\mathbb{Z}^{2}},\overline{\nu},S)$.

\subsection{Completion of the construction}

Let $(\Omega_{0},\mathcal{B}_{0},\mu_{0},T_{0})$ be the action constructed
in Section \ref{sub:a-rank-1-action}. We now augment it to obtain
an action with non-uniform slow entropy $>1$. Informally, we shall
color each point in $A$ randomly by one of two colors $a$ and $b$.
In order to formalize this it is convenient to represent the actions
symbolically.

First, we identify $(\Omega_{0},\mathcal{B}_{0},\mu_{0},T_{0})$ with
a shift invariant measure on $\{0,1\}^{\mathbb{Z}^{2}}$ obtained
by pushing $\mu_{0}$ through the map $f:\Omega_{0}\rightarrow\{0,1\}^{\mathbb{Z}^{2}}$
which codes for the partition $\{A,\Omega\setminus A\}$, \[
f(x)_{u}=1_{A}(T^{u}x)\]
Thus from now on $\Omega_{0}=\{0,1\}^{\mathbb{Z}^{2}}$, $\mu_{0}$
is a shift-invariant measure on $\Omega_{0}$, and $T_{0}=S$, the
shift.

Next, let $a,b$ be new symbols and let $\mu_{1}$ denote the product
(Bernoulli) measure on $\Omega_{1}=\{a,b\}^{\mathbb{Z}^{2}}$ whose
marginals are the uniform $(\frac{1}{2},\frac{1}{2})$ measure on
$\{a,b\}$. The probability measure $\mu_{1}$ is $S$-invariant and
ergodic. 

Let $\Omega_{2}=\{0,1\}^{\mathbb{Z}^{2}}\times\{a,b\}^{\mathbb{Z}^{2}}$,
and consider the product measure $\mu_{2}=\mu_{0}\times\mu_{1}$,
which we regard as a shift-invariant measure on $\Omega_{2}$ in the
obvious manner. 
\begin{claim}
\label{cla:ergodicity-of-product}$(\Omega_{2},\mu_{2},S)$ is ergodic.\end{claim}
\begin{proof}
This follows e.g. from \cite{FurstenbergWeiss1978}: $\mu_{2}$ is
a Bernoulli measure and therefore mildly mixing, so its product with
any ergodic i.m.p. action, and in particular $(\Omega_{2},\mu_{2},S)$
is ergodic.
\end{proof}
We next {}``erase'' the symbol $a,b$ from sites which are not in
$A$. Let $\pi:\{0,1\}\times\{a,b\}\rightarrow\{0,a,b\}$ be the map
\[
\pi(\sigma,\tau)=\left\{ \begin{array}{cc}
0 & \mbox{ if }\sigma=0\\
\tau & \mbox{ if }\sigma=1\end{array}\right.\]
and let $\mu_{3}=\overline{\pi}\mu_{2}$. Clearly, $(\mu_{0},S)$
is a factor of $(\mu_{3},S)$, obtained from the symbol-wise map $\pi'(0)=0$
and $\pi'(a)=\pi'(b)=1$.

Identify the set $A\subseteq\Omega_{0}$ with the set $(\overline{\pi}')^{-1}A\subseteq\Omega_{3}$,
which we also denote by $A$. Let $\mathcal{R}=\{\Omega_{3}\setminus A,A\}$;
then clearly \[
\rho_{slow}(S,\mu_{3},\mathcal{R})=\rho_{slow}(S,\mu_{0},\{\Omega_{0}\setminus A,A\})=0\]
On the other hand,
\begin{claim}
\label{cla:slow-entropy-estimate}$\rho_{slow}(S,\mu_{3})\geq3/2$.\end{claim}
\begin{proof}
Let $\mathcal{P}=\{P_{0},P_{a}P_{b}\}$ denote the partition with
elements are the cylinder sets $P_{\sigma}=\{x\in\Omega_{3}\,:\, x_{0}=\sigma\}$,
$\sigma\in\{0,a,b\}$, so $\mathcal{P}$ generates, and note that
$\mathcal{R}=\{P_{0},P_{a}\cup P_{b}\}$. Our aim is to estimate $N(A,d_{\mathcal{P},n},\varepsilon)$
from below for $n=2r(i)$.

Indeed, fix $n=2r(i)$ and consider the partition of $A$ induced
by $\mathcal{P}^{Q_{n}}$, which refines $\mathcal{R}^{Q_{n}}$. Let
$A_{1},\ldots,A_{M}\subseteq A$ denote the intersection of $A$ with
the atoms of $\mathcal{P}^{Q_{n}}$. Since $N(A,d_{\mathcal{P},n},\varepsilon)\leq C$
implies $N(A_{j},d_{\mathcal{P},n},\varepsilon)\leq C$ for some $j$,
it suffice to show for every $j$ that this is impossible if $C\leq2^{r(i)^{3/2+o(1)}}$. 

Fix $1\leq j\leq M$, let $x_{j}\in A_{j}$, and note that $R_{n}(A,x_{j})\subseteq Q_{r}$
depends on $j$ but not on the choice of $x_{j}\in A_{j}$. By Claim
\ref{cla:recurrence-rate-bound}, we have $|R_{n}(A,x_{j})|=|Q_{n}|^{3/4+o(1)}$.

From the construction of $\mu_{3}$ it is easy to see that the $(\mathcal{P},n)$-names
arising from points in $A_{j}$ consist of all colorings of $Q_{r}$
such that $Q_{r}\setminus R_{n}(A,x_{j})$ is colored $0$ and $R_{n}(A,x_{j})$
is colored by $a$ and $b$. Furthermore these colorings are equally
likely with respect to $\mu_{3}|_{A_{j}}$. Thus, $N(A_{j},d_{\mathcal{P},n},\varepsilon)=N(\{a,b\}^{R_{n}(A,x_{j})},d,\varepsilon)$,
where $d$ is the standard Hamming distance and the measure on $\{a,b\}^{R_{n}(A,x_{j})}$
is the product measure with uniform marginals. It is well-known, however,
that there is a lower bound of the form $2^{(1-\delta)|R_{n}(A,x_{j})|}$,
where $\delta=\delta(\varepsilon)\rightarrow0$ as $\varepsilon\rightarrow0$.
This gives $N(A_{j},d_{\mathcal{P},r},\varepsilon)\geq2^{r^{3/2+o(1)}}$,
which is the desired result.
\end{proof}
In summary, we have shown that $(\Omega_{3},\mu_{3},S)$ has slow
entropy $\geq3/2$ but is does not have uniform slow entropy, because
with respect to $\mathcal{R}$ its slow entropy is $0$. Thus by Theorem
\ref{thm:large-s-implies-no-diff-model} it has no differentiable
model. This establishes Theorem \ref{thm:non-realizable actions-exist}.

\bibliographystyle{plain}
\bibliography{bib}

\end{document}